\documentclass{amsart}
\usepackage{amssymb,amsmath,amsthm,latexsym,booktabs,units, todonotes,color,enumerate,comment,lineno, graphics, mathtools}
\usepackage[hidelinks]{hyperref}
\usepackage[a4paper,margin=1.15in]{geometry}
\theoremstyle{plain}
\newtheorem{theorem}{Theorem}
\newtheorem{corollary}{Corollary}

\newtheorem{lemma}{Lemma}
\newtheorem{fact}{Fact}
\newtheorem{proposition}{Proposition}
\newtheorem{conjecture}{Conjecture}

\theoremstyle{definition}
\newtheorem{definition}{Definition}
\newtheorem{example}{Example}

\theoremstyle{remark}

\newtheorem{remark}{\underline{Remark}}

\newcommand{\bit}{\begin{itemize}}
\newcommand{\eit}{\end{itemize}}
\newcommand{\ben}{\begin{enumerate}}
\newcommand{\een}{\end{enumerate}}
\newcommand{\beq}{\begin{equation}}
\newcommand{\eeq}{\end{equation}}
\newcommand{\bea}{\begin{eqnarray*}}
\newcommand{\eea}{\end{eqnarray*}}
\newcommand{\bpf}{\begin{proof}}
\newcommand{\epf}{\end{proof}}

\usepackage{todonotes}

\begin{document}

\title[Combinatorial Nullstellensatz and graph labelings]{Application of the Combinatorial Nullstellensatz\\ to magic-type graph labelings
}

\author{Parikshit Chalise}
\address{Department of Applied Mathematics and Statistics\\
        Johns Hopkins University\\
        Baltimore, MD 21218\\
         USA}
         \email{pchalis1@jhu.edu}
\author{Richard M. Low}
\address{Department of Mathematics and Statistics\\
         San Jose State University\\
         San Jose, CA 95192\\
         USA}
\email{richard.low@sjsu.edu}


\date{{July 8, 2026.}}

\begin{abstract}
Let $G=(V,E)$ be a simple graph, and let $k\geq 2$ be an integer. For an edge
labeling $h:E(G)\to \mathbb{Z}_{k} \backslash \{0\}$, define the induced vertex label
by
\[
h^+(v)=\sum_{e \ni v} h(e) \pmod{k}.
\]
For $t\in \mathbb Z_k$, we say that $G$ is \emph{$t$-sum $\mathbb Z_k$-magic}
if there exists such a labeling $h$ satisfying
\[
h^+(v)=t
\qquad\text{for all }v\in V.
\]
We say that $G$ is \emph{$\mathbb Z_k$-magic} if $G$ is $t$-sum
$\mathbb Z_k$-magic for some $t\in \mathbb Z_k$. Similarly, if there exists an edge labeling $h: E(G) \to \mathbb{Z}_{k} \backslash \{0\}$ such that the induced vertex labeling $h^+(v)=\sum_{e\ni v} h(e)$ (mod $k$) is injective, then $G$ is called \emph{$\mathbb{Z}_{k}$-antimagic}. In this paper, we use the Combinatorial Nullstellensatz to analyze these two types of magic graph labelings.
\end{abstract}
\maketitle


\section{Introduction}\label{S:intro}
Graph labelings occupy a prominent place in graph theory. Their formal introduction is attributed to Rosa's work from the 1960s \cite{Ro}. Since then, graph labelings have drawn the attention of many mathematicians. Gallian's survey \cite{Gallian}, which contains 4,045 references, testifies to the popularity of the subject. Beyond their intrinsic mathematical appeal, graph labelings have found diverse applications in coding theory, cryptology, circuit design, communication network addressing, database management, astronomy, and many other areas; see Bloom and Golomb \cite{BG2, BG1}, the introduction of \cite{Gallian}, and the references therein.

A large literature on graph labelings has been devoted to the explicit construction of desired labelings. Consequently, these results are often limited to narrow families of graphs. In this paper, we use the Combinatorial Nullstellensatz to demonstrate the existence or nonexistence of integer-magic and integer-antimagic labelings. Our main results are derived by associating monomials occurring in the expanded form of a suitable multivariate polynomial with structural properties of graphs.

\section{Preliminaries} \label{S:Prelim}
In \cite{Alon99}, Alon proved the following important result and presented several applications to problems in combinatorics, number theory, and graph theory.

\begin{theorem} \label{CN} $($Combinatorial Nullstellensatz \cite{Alon99}$)$.
Let $f( x) = f(x_1, \dots, x_m)$ be a polynomial of degree $d$ over a field $\mathbb{F}$. Suppose that the coefficient of the monomial $x_1^{t_1} \cdots x_m^{t_m}$ in $f$ is nonzero and $t_1 + \cdots + t_m = d$. If $S_1, \dots, S_m$ are subsets of $\mathbb{F}$ with $|S_i| \geq t_i + 1$, then there exists an ${s} = (s_{1},\dots,s_{m}) \in S_1 \times \cdots \times S_m$ for which $f(s) \neq 0$.
\end{theorem}

\noindent
Various applications of Theorem \ref{CN} to graph colorings already appeared in Alon's original paper \cite{Alon99}. Since then, the Combinatorial Nullstellensatz has been used to prove results on integer-magic labelings, antimagic labelings, neighbor-sum distinguishing total labelings, and list colorings \cite{Hefetz, Huang_Wong_Zhu,  Low_Roberts_Application,Low_Roberts_Constructing, Prz, USch2010}. Theorem \ref{CN} has also been used to obtain graph decomposition results arising from associated graph labeling problems \cite{Camara_Llado_Moragas, Kezdy, Llado} (see also Chapter 7 of \cite{Lopez_MuntanerBatle}). See \cite{Zhu_Balakrishnan} for a recent monograph on the Combinatorial Nullstellensatz and graph coloring problems.

Throughout the paper, for an integer $k\geq 2$, we write $\mathbb Z_k$ for
the set $\{0,1,\ldots,k-1\}$, with addition and multiplication taken modulo
$k$. We write $\mathbb Z_k^*=\mathbb Z_k\setminus\{0\}$. When $p$ is prime, $\mathbb Z_p$ is the finite
field of integers modulo $p$. For a positive integer $n$, we write
$[n]=\{1,2,\ldots,n\}$.

\begin{definition}\label{def:Z_n-magic}
Let $G=(V,E)$ be a simple graph, and let $k\geq 2$ be an integer. For an edge
labeling $h:E(G)\to \mathbb Z_k\setminus\{0\}$, define the induced vertex label
by
\[
h^+(v)=\sum_{e \ni v} h(e) \pmod{k}.
\]
For $t\in \mathbb Z_k$, we say that $G$ is \emph{$t$-sum $\mathbb Z_k$-magic}
if there exists such a labeling $h$ satisfying
\[
h^+(v)=t
\qquad\text{for all }v\in V.
\]
We say that $G$ is \emph{$\mathbb Z_k$-magic} if $G$ is $t$-sum
$\mathbb Z_k$-magic for some $t\in \mathbb Z_k$. Similarly, if there exists an edge labeling $h: E(G) \to \mathbb{Z}_{k} \backslash \{0\}$ such that the induced vertex labeling $h^+(v)=\sum_{e\ni v} h(e)$ (mod $k$) is injective, then $G$ is called \emph{$\mathbb{Z}_{k}$-antimagic}. 
\end{definition}

\begin{example}
    Figure \ref{fig:first-examples} shows examples of a $3$-sum $\mathbb Z_5$-magic labeling and a $\mathbb Z_5$-antimagic labeling.
\end{example}

\begin{figure}[h]
\centering

\begin{minipage}{0.48\textwidth}
\centering
\begin{tikzpicture}[
    scale=1,
    every node/.style={font=\small},
    elab/.style={}
]
    \node[circle,draw,inner sep=2pt] (v1) at (0,0) {$\color{red}3$};
    \node[circle,draw,inner sep=2pt] (v2) at (-1.5,-1.8) {$\color{red}3$};
    \node[circle,draw,inner sep=2pt] (v3) at (1.5,-1.8) {$\color{red}3$};
    \node[circle,draw,inner sep=2pt] (v4) at (-0.9,1.8) {$\color{red}3$};
    \node[circle,draw,inner sep=2pt] (v5) at (0.9,1.8) {$\color{red}3$};

    \draw (v1) -- (v2);
    \draw (v1) -- (v3);
    \draw (v1) -- (v4);
    \draw (v1) -- (v5);
    \draw (v2) -- (v3);

    \node[elab] at (-0.95,-0.8) {$\color{blue}1$};
    \node[elab] at (0.95,-0.8) {$\color{blue}1$};
    \node[elab] at (-0.7,0.95) {$\color{blue}3$};
    \node[elab] at (0.7,0.95) {$\color{blue}3$};
    \node[elab] at (0,-2.05) {$\color{blue}2$};
\end{tikzpicture}

\vspace{2mm}
$\mathbb Z_5$-magic labeling
\end{minipage}
\hfill
\begin{minipage}{0.48\textwidth}
\centering
\begin{tikzpicture}[
    scale=1,
    every node/.style={font=\small},
    elab/.style={}
]
    \node[circle,draw,inner sep=2pt] (w1) at (0,0) {$\color{red}2$};
    \node[circle,draw,inner sep=2pt] (w2) at (-1.5,-1.8) {$\color{red}4$};
    \node[circle,draw,inner sep=2pt] (w3) at (1.5,-1.8) {$\color{red}0$};
    \node[circle,draw,inner sep=2pt] (w4) at (-0.9,1.8) {$\color{red}1$};
    \node[circle,draw,inner sep=2pt] (w5) at (0.9,1.8) {$\color{red}3$};

    \draw (w1) -- (w2);
    \draw (w1) -- (w3);
    \draw (w1) -- (w4);
    \draw (w1) -- (w5);
    \draw (w2) -- (w3);

    \node[elab] at (-0.95,-0.8) {$\color{blue}1$};
    \node[elab] at (0.95,-0.8) {$\color{blue}2$};
    \node[elab] at (-0.7,0.95) {$\color{blue}1$};
    \node[elab] at (0.7,0.95) {$\color{blue}3$};
    \node[elab] at (0,-2.05) {$\color{blue}3$};
\end{tikzpicture}

\vspace{2mm}
$\mathbb Z_5$-antimagic labeling
\end{minipage}

\caption{A graph that is both $\mathbb Z_5$-magic and $\mathbb Z_5$-antimagic. The blue labels are edge labels, and the red labels are induced vertex labels.} \label{fig:first-examples}
\end{figure}
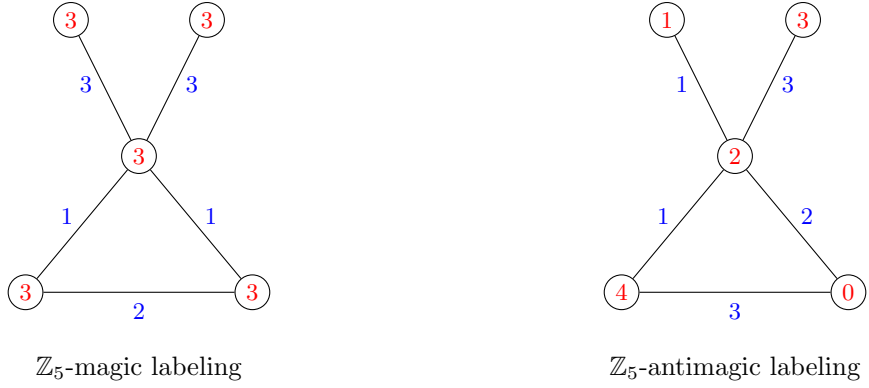

\section{$\mathbb{Z}_k$-magic labelings} \label{S:magic}

Let $[G, \mathbb{Z}_k]$ denote
the set of distinct $\mathbb{Z}_k$-magic labelings of a graph $G$. The graph $G$ is
$\mathbb{Z}_k$-magic if and only if $[G, \mathbb{Z}_k] \neq \varnothing$. For any commutative ring $R$ with
unity, let $U(R)$ denote the multiplicative group of units in $R$. Theorem \ref{GroupAction} establishes a group action on $[G, \mathbb{Z}_k]$.

\begin{theorem}\label{GroupAction} $($\textit{Low-Lee} \cite{Low_Lee_Eulerian}$)$. Let $d
\in U(\mathbb{Z}_k)$ and $f \in [G, \mathbb{Z}_k]$. Then, $d\cdot f \in [G, \mathbb{Z}_k]$.
\end{theorem}

\indent
We also note the following useful lemma.

\begin{lemma} \label{lem:GM_divide}
$($\textit{Low-Lee} \cite{Low_Lee_Eulerian}$)$. Let $G$ be a $\mathbb Z_k$-magic graph with $k \mid n$. Then, $G$ is a $\mathbb Z_n$-magic graph.
\end{lemma}

\begin{remark}
    The converse of Lemma \ref{lem:GM_divide} is not true in general. For example, $K_4 - \{e\}$ is $\mathbb{Z}_6$-magic but it is not $\mathbb{Z}_{3}$-magic.
\end{remark}

\begin{remark}
Lemma \ref{lem:GM_divide} remains true if one replaces ``$\mathbb{Z}_k$-magic" by ``$\mathbb{Z}_k$-antimagic." 
\end{remark}

We will repeatedly use the following two well-known facts throughout this paper.

\begin{fact} \label{fact:multinomial}
\textit{$($Multinomial theorem$)$}. For variables $x_1,\ldots,x_m$ and an integer $r\geq 0$,
\[
(x_1+\cdots+x_m)^r
=
\sum_{\substack{a_1,\ldots,a_m\geq 0\\ a_1+\cdots+a_m=r}}
\binom{r}{a_1,\ldots,a_m}
x_1^{a_1}\cdots x_m^{a_m},
\]
where
\[
\binom{r}{a_1,\ldots,a_m}
=
\frac{r!}{a_1!\cdots a_m!}.
\]
Equivalently, the coefficient of $x_1^{a_1}\cdots x_m^{a_m}$ in
$(x_1+\cdots+x_m)^r$ is $\frac{r!}{a_1!\cdots a_m!}
$ whenever $a_1+\cdots+a_m=r$, and it is $0$ otherwise.
\end{fact}

\begin{fact} \label{fact:Wilson}
\textit{$($Wilson's theorem$)$}. If $p$ is prime, then $(p-1)! \equiv -1 \pmod p.$
\end{fact}

\subsection{Hartke $\mathbb{Z}_p$-magic graphs}\label{Sect_3.1}

Consider a simple graph $G = (V, E)$, where $|E| = m$. Let the edges of $G$ be identified with distinct variables $x_1, x_2, \dots, x_m$. In view of Lemma \ref{lem:GM_divide}, we focus on $\mathbb{Z}_p$-magic graphs, where $p$ is prime. The characterization of $\mathbb{Z}_2$-magic graphs is complete (see Lemma 2.1 in \cite{Low_Roberts_Constructing}). Hence, we consider primes $p \geq3$.  For fixed prime $p\geq 3$ and $t \in \mathbb{Z}_p$, define the polynomial $f_{G,t}$ in $\mathbb{Z}_p[x_1, \dots, x_m]$ as
\begin{equation} \label{HartkeEq}
f_{G,t}(x)=f_{G,t}(x_{1},\dots,x_{m})=\prod_{v\in V(G)}\left[1-\left(t-\sum_{e\ni v}x_{e}\right)^{p-1}\right],
\end{equation}
where addition and multiplication are taken modulo $p$. The factorization displayed in  (\ref{HartkeEq}) is called the \textit{canonical factorization} of $f_{G,t}$, and its factors are called the \textit{canonical factors}. The polynomial $f_{G,t}$ is called the \textit{Hartke polynomial} and was introduced in \cite{Low_Roberts_Constructing}. Observe that a graph $G$ is $t$-sum $\mathbb{Z}_p$-magic if and only if there exists $s\in (\mathbb{Z}_p^*)^m$ such that $f_{G,t}(s) \neq 0$. 

\begin{definition}
Let $p$ be prime. A graph $G$ is called \textit{Hartke} $\mathbb{Z}_p$-magic \underline{if} the Combinatorial Nullstellensatz can be applied to a Hartke polynomial $f_{G,t}$ of $G$ to prove that $G$ is $\mathbb{Z}_p$-magic for some $t \in \mathbb{Z}_p$. Equivalently, $f_{G,t}$ has a monomial $M$ of degree $(p-1)|V(G)|$ whose coefficient is nonzero and whose exponent in each variable is at most $p-2$. In this case, such a monomial $M$ is called a \textit{Hartke term}.
\end{definition}

\begin{example}\label{Z_3-base-case}
        Consider the graph $G$ with 6 vertices and 12 edges shown in Figure \ref{fig:Z_3-base-case}. We determine whether it is Hartke $\mathbb{Z}_3$-magic. The corresponding Hartke polynomial is of degree $2\cdot 6= 12$. Here, a desired Hartke term would be of degree at most 1 in each variable. A computation shows that there exists a monomial $x_1x_2\cdots x_{12}$  of degree 12 with coefficient $2^6 \cdot 32 \not \equiv 0 \pmod{3}$. Thus, $G$ is Hartke $\mathbb{Z}_3$-magic.
\end{example}
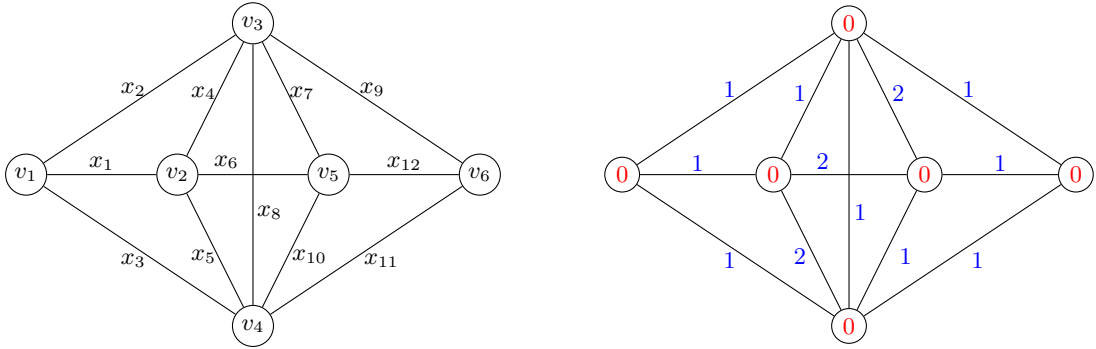
\begin{figure}[h]
\begin{center}
\begin{minipage}{0.48\textwidth}
\begin{center}
\begin{tikzpicture}[
    scale=1.0,
    every node/.style={font=\small},
    vertex/.style={circle, draw, inner sep=2pt},
    elab/.style={}
]

  \node[vertex]  (v1) at (0,0) {$v_1$};
  \node[vertex]  (v2) at (2,0) {$v_2$};
  \node[vertex]  (v3) at (3,2) {$v_3$};
  \node[vertex]  (v4) at (3,-2) {$v_4$};
  \node[vertex]  (v5) at (4,0) {$v_5$};
  \node[vertex]  (v6) at (6,0) {$v_6$};

  \draw (v1) -- (v2);
  \draw (v1) -- (v3);
  \draw (v1) -- (v4);

  \draw (v2) -- (v3);
  \draw (v2) -- (v4);
  \draw (v2) -- (v5);

  \draw (v3) -- (v5);
  \draw (v3) -- (v4);
  \draw (v3) -- (v6);

  \draw (v4) -- (v5);
  \draw (v4) -- (v6);

  \draw (v5) -- (v6);

  \node[elab] at (1,0.15) {$x_1$}; 
  \node[elab] at (1.42,1.13) {$x_2$}; 
  \node[elab] at (1.42,-1.13) {$x_3$}; 

  \node[elab] at (2.35,1.08) {$x_4$}; 
  \node[elab] at (2.35,-1.08) {$x_5$}; 
  \node[elab] at (2.65,0.17) {$x_6$}; 

  \node[elab] at (3.65,1.08) {$x_7$}; 
  \node[elab] at (3.22,-0.5) {$x_8$}; 
  \node[elab] at (4.58,1.13) {$x_9$}; 

  \node[elab] at (3.75,-1.08) {$x_{10}$}; 
  \node[elab] at (4.7,-1.13) {$x_{11}$}; 

  \node[elab] at (5,0.15) {$x_{12}$}; 
\end{tikzpicture}

\end{center}
\end{minipage}
\hfill
\begin{minipage}{0.48\textwidth}
\begin{center}
\begin{tikzpicture}[
    scale=1.0,
    every node/.style={font=\small},
    vertex/.style={circle, draw, inner sep=2pt},
    elab/.style={}
]

  \node[vertex] (w1) at (0,0) {$\color{red}0$};
  \node[vertex] (w2) at (2,0) {$\color{red}0$};
  \node[vertex] (w3) at (3,2) {$\color{red}0$};
  \node[vertex] (w4) at (3,-2) {$\color{red}0$};
  \node[vertex] (w5) at (4,0) {$\color{red}0$};
  \node[vertex] (w6) at (6,0) {$\color{red}0$};

  \draw (w1) -- (w2);
  \draw (w1) -- (w3);
  \draw (w1) -- (w4);

  \draw (w2) -- (w3);
  \draw (w2) -- (w4);
  \draw (w2) -- (w5);

  \draw (w3) -- (w5);
  \draw (w3) -- (w4);
  \draw (w3) -- (w6);

  \draw (w4) -- (w5);
  \draw (w4) -- (w6);

  \draw (w5) -- (w6);

  \node[elab] at (1,0.15) {$\color{blue}1$}; 
  \node[elab] at (1.42,1.13) {$\color{blue}1$}; 
  \node[elab] at (1.42,-1.13) {$\color{blue}1$}; 

  \node[elab] at (2.35,1.08) {$\color{blue}1$}; 
  \node[elab] at (2.35,-1.08) {$\color{blue}2$}; 
  \node[elab] at (2.65,0.17) {$\color{blue}2$}; 

  \node[elab] at (3.65,1.08) {$\color{blue}2$}; 
  \node[elab] at (3.15,-0.5) {$\color{blue}1$}; 
  \node[elab] at (4.58,1.13) {$\color{blue}1$}; 

  \node[elab] at (3.75,-1.08) {$\color{blue}1$}; 
  \node[elab] at (4.7,-1.13) {$\color{blue}1$}; 

  \node[elab] at (5,0.15) {$\color{blue}1$}; 
\end{tikzpicture}
\end{center}
\end{minipage}

\end{center}
\caption{A Hartke $\mathbb{Z}_3$-magic graph $G$, along with a 0-sum $\mathbb{Z}_3$-magic labeling.}
\label{fig:Z_3-base-case}
\end{figure}

The motivation for separating the class of Hartke $\mathbb{Z}_p$-magic graphs from general $\mathbb{Z}_p$-magic graphs arises from the special properties of this class. We summarize them below:

\begin{theorem} \label{thm:hartke-special}
$($\textit{Low-Roberts} \cite{Low_Roberts_Constructing}$)$. Let $p \geq 3$ be prime. Suppose $G$ is a Hartke $\mathbb{Z}_p$-magic graph. Then,
\begin{enumerate}
    \item {The graph $G$ is t-sum Hartke $\mathbb{Z}_p$-magic for all $t\in \mathbb{Z}_p$.}
    \item A graph $G'$ formed by adding simple edges to $G$ is also Hartke $\mathbb{Z}_p$-magic.
\item For every graph $H$, the graphs $G \square H$, $G\circ H$, $H\circ G$, and $G \boxtimes H$ are $\mathbb{Z}_p$-magic, where $\square$, $\circ$, and $\boxtimes$ denote the Cartesian, lexicographic, and strong products, respectively. 
\end{enumerate}
\end{theorem}

\begin{remark} \label{rmk:edge-stability-c4+e}
    Statement (2) of Theorem \ref{thm:hartke-special} is called the \textit{edge-stability} property in \cite{Low_Roberts_Constructing}. Note that the edge-stability property does not hold for general $\mathbb{Z}_p$-magic graphs. For instance, $C_4$ is $\mathbb{Z}_3$-magic, but $C_4+\{e\}$ is not.
\end{remark}

We now establish some immediate necessary conditions for a graph $G$ to be Hartke $\mathbb{Z}_p$-magic. The following is a consequence of the degree restriction imposed by the  Combinatorial Nullstellensatz.

\begin{proposition} \label{Prop_1}
$($\textit{Low-Roberts} \cite{Low_Roberts_Constructing}$)$. Let $p \geq 3$ be prime. If $G$ is Hartke $\mathbb{Z}_p$-magic, then $|E(G)| \geq$ $\frac{p-1}{p-2} \cdot |V(G)|$. 
\end{proposition}

\noindent
\begin{remark} \label{rmk:hartke-cycle-not-applicable}
If a graph satisfies the inequality in Proposition \ref{Prop_1} for $p=3$, then it also satisfies the inequality for all primes $p\geq 5$. However, the edge-count bound in Proposition \ref{Prop_1}  does not allow us to analyze cycle graphs for any $p$; we address this limitation in  Section \ref{sec:alon-magic}.
\end{remark}

The following result gives another necessary condition for $G$ to be Hartke $\mathbb{Z}_p$-magic. We use the standard notation $\delta(G)$ to denote the minimum degree of the graph $G$.

\begin{proposition}\label{Nec_Hartke_B}
Let $G$ be a connected graph. If $G$ is Hartke $\mathbb{Z}_p$-magic, then $\delta(G) \geq 2$.   
\end{proposition}
\begin{proof}
This is equivalent to proving that if $\delta(G) = 1$, then $G$ is not Hartke $\mathbb{Z}_p$-magic. We proceed by contradiction. Let $v \in V(G)$ be of degree one and assume that $G$ is Hartke $\mathbb{Z}_p$-magic. Then, there are $t$-sum $\mathbb{Z}_p$-magic labelings of $G$ for all $t\in \mathbb{Z}_p$, by Theorem \ref{thm:hartke-special}. However, there cannot exist a $\mathbb{Z}_p$-magic labeling in which $v$ has magic value $0$, since that would force the single edge incident to $v$ to be labeled 0. This yields the desired contradiction. Hence, the proposition is established.
\end{proof}

The following conjecture was posed in \cite{Low_Roberts_Constructing}. 
\begin{conjecture} \label{Conj1}
Let $p \geq 3$ be prime. For every integer $n \geq 6$, there exists a connected simple Hartke $\mathbb{Z}_p$-magic graph of order $n$.
\end{conjecture}

Next, we develop a general technique that allows one to construct Hartke $\mathbb{Z}_p$-magic graphs from existing Hartke $\mathbb{Z}_p$-magic graphs.

\begin{proposition}\label{prop:induction}
Let $p\geq 3$ be prime. Suppose that a graph $G$ is Hartke $\mathbb Z_p$-magic. Let $G^+$ be obtained from $G$ by adding a new vertex
adjacent to exactly $p-1$ distinct vertices of $G$. Then,  $G^+$ is Hartke
$\mathbb Z_p$-magic.
\end{proposition}

\begin{proof}
Suppose that $G$ is Hartke $\mathbb Z_p$-magic. Then there is a monomial $M$
appearing in the Hartke polynomial $f_{G,0}$ with a nonzero coefficient,\footnote{It suffices to consider $f_{G,t}$ with $t=0$ by statement (1) of Theorem \ref{thm:hartke-special}.} where the exponent of every variable is at most $p-2$ and $\deg M=(p-1)|V(G)|$. Suppose that we add a new vertex $w$, and let the new edges incident to $w$ be $y_1,\ldots,y_{p-1}$. The $0$-sum Hartke
polynomial of $G^+$ is
\[
f_{G^{+},0}
=
\underbrace{
\prod_{v\in V(G)}
\left[
1-\left(0-
\sum_{e\ni v}x_e-\sum_{y_i\ni v}y_i
\right)^{p-1}
\right]
}_{P(x,y)}
\underbrace{
\left[
1-\left(0-\sum_{i=1}^{p-1}y_i\right)^{p-1}
\right]
}_{Q(y)}.
\]
Consider the monomial
\[
M^+=M y_1y_2\cdots y_{p-1}.
\]
Observe that the degree of $M^+$ is $(p-1)|V(G)|+(p-1)=(p-1)|V(G^+)|$, and each variable appearing in $M^+$ has exponent at most $p-2$. Therefore, it suffices to prove that $M^+$ has nonzero coefficient in $f_{G^+,0}$. 

First, since $P(x,y)$ is a product of $|V(G)|$ factors, each of degree at most
$p-1$, we have
\[
\deg P\leq (p-1)|V(G)|.
\]
But $\deg M=(p-1)|V(G)|$. Therefore, if a term of $P(x,y)$ contains the original 
monomial $M$, then it cannot contain a positive power of the new variables $y_1,\ldots,y_{p-1}$. Hence the coefficient of $ M^+ = M y_1\cdots y_{p-1}$ is determined by 
the coefficient of $M$ in $P(x,0)=f_{G,0}$ and the coefficient of
$y_1\cdots y_{p-1}$ in $Q(y)$.
By the multinomial theorem (Fact 1), the coefficient of $y_1\cdots y_{p-1}$ in $Q(y)$ is (up to sign) equal to $(p-1)! \equiv -1 \pmod{p}.$
Therefore, the coefficient of the monomial $M^+$ in $f_{G^+,0}$ is  (up to sign) congruent to the coefficient of the monomial $M$ in $f_{G,0}$. In particular, $M^+$ has a nonzero coefficient in $f_{G^+,0}$.  Thus, $G^{+}$ is Hartke $\mathbb{Z}_p$-magic.
\end{proof}

\begin{theorem}\label{Hartke_Z3_A}
For every $n \geq 6$, there exists a connected simple Hartke $\mathbb{Z}_3$-magic graph of order $n$.
\end{theorem}

\begin{proof} 
Let $G_6$ be the connected simple graph on $6$ vertices and $12$ edges from Example \ref{Z_3-base-case}, which is known to be Hartke $\mathbb{Z}_3$-magic. 
We now construct graphs $G_n$ for every $n > 6$ recursively according to Proposition \ref{prop:induction}. This completes the proof.
\end{proof}

\begin{definition}\label{def:Nk-orientations}
Let $G$ be a graph, and let $p \geq 3$ be prime. For $k\in \mathbb Z_p$, we denote by $N_k(G)$ the number of orientations of $G$ in which every
vertex has indegree ${k}$.
\end{definition}

The following result provides a combinatorial criterion for a graph to be Hartke $\mathbb Z_p$-magic.

\begin{theorem}\label{thm:orientation-certificate}
Let $p\geq 3$ be prime. Suppose that a graph $G$ contains a spanning subgraph $J$ such that $|E(J)|=(p-1)|V(G)|.$ If $N_{p-1}(J)\not\equiv 0 \pmod p,$ then $G$ is Hartke $\mathbb Z_p$-magic.
\end{theorem}
\begin{proof}
Let $J$  be a spanning subgraph of a graph $G$ such that $|E(J)|=(p-1)|V(G)|.$ Consider the monomial
\[
M=\prod_{e\in E(J)} x_e.
\]
Observe that
\[
\deg M=|E(J)|=(p-1)|V(G)|=\deg f_{G,0},
\]
where $f_{G,0}$ is the Hartke polynomial 
\begin{equation} \label{eqn:Hartke-poly-lem2}
f_{G,0}=\prod_{v\in V(G)}\left[1-\left(0-\sum_{e\ni v}x_{e}\right)^{p-1}\right].
\end{equation}
Further, the degree of each variable in $M$ is at most 1 $\leq p-2$. Hence, if the coefficient
of $M$ in $f_{G,0}$ is not $0$ modulo $p$, then  $G$ is Hartke $\mathbb Z_p$-magic. We now interpret
this coefficient combinatorially. We show that
the coefficient of $M$ in $f_{G,0}$ is (up to sign) equal to $N_{p-1}(J)$, the number of orientations of $J$ in which every
vertex has indegree $p-1$.

To obtain the monomial $M$ as we expand $f_{G,0}$, each factor $\big(\sum_{e\ni v}x_e\big)^{p-1}$ must contribute a product
of $p-1$ distinct incident edge variables from $E(J)$. For a fixed unordered set of $p-1$ distinct incident edges
$\{e_1,\dots,e_{p-1}\}$ at vertex $v$, the monomial
$x_{e_1}\cdots x_{e_{p-1}}$ appears in
$\big(\sum_{e\ni v}x_e\big)^{p-1}$ with coefficient
$(p-1)!\equiv -1\pmod p$, by Facts \ref{fact:multinomial} and \ref{fact:Wilson}.  
Such a choice is equivalent to choosing $p-1$
distinct incident edges of $J$ for each vertex so that every edge of $J$ is chosen at exactly one of its endpoints. Given such a choice, orient each edge toward the endpoint at which it was chosen. Then, every vertex has exactly
$p-1$ incoming edges. Conversely, from any orientation of $J$ with indegree
$p-1$ at every vertex, choose at each vertex its $p-1$ incoming edges. Thus, the coefficient of $M$ in $f_{G,0}$ is (up to sign) $N_{p-1}(J)$, as claimed.

It follows that $G$ is Hartke
$\mathbb Z_p$-magic whenever $N_{p-1}(J) \not\equiv 0 \pmod{p}$.
\end{proof}

We note that the claim of Theorem \ref{thm:orientation-certificate} becomes an ``if and only if'' statement when $p = 3$, because in this case the only possible Hartke term is of the form described in the proof. Thus, we obtain the following corollary, which gives a complete characterization of Hartke $\mathbb Z_3$-magic graphs in terms of orientation counts.

\begin{corollary}\label{cor:hartke-z3}
A graph $G$ is Hartke $\mathbb Z_3$-magic if and only if it contains a spanning subgraph $J$ such that $N_2(J) \not\equiv 0 \pmod{3}$.
\end{corollary}

\begin{remark}
    In principle, one could use Theorem \ref{thm:orientation-certificate} to construct Hartke $\mathbb{Z}_p$-magic graphs. However, checking the hypothesis of Theorem \ref{thm:orientation-certificate} becomes time-consuming once $n$ is sufficiently large.
\end{remark}

In work that preceded the Combinatorial Nullstellensatz, Alon and Tarsi \cite{AlonTarsi92} had already used the polynomial method to obtain results on graph coloring topics, including questions involving the chromatic number and the choice number of graphs. Their main contribution was to relate coefficients of a polynomial to the number of Eulerian orientations. This approach of enumerating Eulerian orientations also led to a proof of the ``cycle-plus-triangle problem" \cite{FleischnerStiebitz92} posed by Erdős. Inspired by the Alon-Tarsi approach, we next establish a technique that utilizes the number of Eulerian orientations to certify the existence of Hartke $\mathbb Z_p$-magic graphs.

\begin{definition}
A graph $G$ is called \emph{Eulerian} if every vertex of $G$ has even degree. An \emph{Eulerian orientation} of $G$ is an orientation of its edges such that the indegree $d^-(v)$ equals the outdegree $d^+(v)$ at every vertex $v \in V(G)$.
\end{definition}

\begin{definition}
A spanning subgraph $F$ of a graph $G$ is called a \textit{$k$-factor} if every vertex has degree $k$ in $F$.
\end{definition}
\begin{lemma}\label{lem:factor-eulerian-coeff}
Let $p\geq 5$ be prime. Suppose that a graph $G$ contains a
$(p-1)$-factor $F$. Define
\[
\Psi_p(F)
=
\sum_{\substack{K\subseteq F\\ K\text{ Eulerian}}}
EO(K)\,2^{-|E(K)|},
\]
where $EO(K)$ denotes the number of Eulerian orientations of $K$. If
$\Psi_p(F)\not\equiv 0\pmod{p}$, then $G$ is Hartke
$\mathbb Z_p$-magic.
\end{lemma}
\begin{proof}
Consider the monomial
\[
M=\prod_{e\in E(F)} x_e^2.
\]
Since $F$ is spanning and $(p-1)$-regular,
\[
\deg M=2|E(F)|=(p-1)|V(G)| = \deg f_{G,0},
\]
where $f_{G,0}$ is the corresponding Hartke polynomial of $G$. Moreover, every exponent of $M$ is at most $2 \leq p-2$, as needed for a Hartke term. Hence, if the coefficient
of $M$ in $f_{G,0}$ is not $0$ modulo $p$, then  $G$ is Hartke $\mathbb Z_p$-magic. We now interpret
this coefficient combinatorially. We show that
the coefficient of $M$ in $f_{G,0}$ is (up to sign) equal to $\Psi_p(F)$.

Since $M$ contains no variable corresponding to an edge outside $F$, the relevant coefficient is the coefficient of $M$ in
\begin{equation}\label{eq:polynomial-for-M}
    \prod_{v\in V(F)}
\bigg(\sum_{\substack{e\ni v\\ e\in E(F)}}x_e\bigg)^{p-1}.
\end{equation}
For the factor indexed by $v\in V(F)$, let $a_{v,e}$ denote the exponent of
$x_e$ contributed by that factor. Thus,
\begin{equation}\label{eq:sum_v}
    \sum_{e\ni v}a_{v,e}=p-1
\quad\text{for all }v\in V(F),
\end{equation}
and since the monomial $M$ has exponent $2$ on each edge of $F$,
\begin{equation}\label{eq:sum_e}
    a_{u,e}+a_{v,e}=2
\quad\text{for all }e=uv\in E(F).
\end{equation}
Fix an arbitrary orientation of $F$. If $e$ is oriented from $u$ to $v$, define
$\psi(e)=a_{u,e}-1$. By \eqref{eq:sum_e}, we also have
$\psi(e)=1-a_{v,e}$, and hence $\psi(e)\in\{-1,0,1\}$. Then \eqref{eq:sum_v} is equivalent to
\begin{equation}\label{eq:orient_sum}
\sum_{e\in E^+(v)}\psi(e)
=
\sum_{e\in E^-(v)}\psi(e) 
\quad\text{for every }v\in V(F),
\end{equation}
where $E^+(v)$ and $E^-(v)$ denote, respectively, the sets of outgoing and incoming edges at vertex $v$. Let $K_\psi=\{e\in E(F):\psi(e)\neq 0\}$. We claim that $K_\psi$ is an
Eulerian subgraph of $F$. Orient each edge $e\in K_\psi$ by keeping the
orientation fixed when $\psi(e)=1$ and reversing it when $\psi(e)=-1$. In this orientation of $K_\psi$, the
indegree and outdegree are equal at every vertex by \eqref{eq:orient_sum}. Conversely, every Eulerian orientation of an Eulerian subgraph $K\subseteq F$
arises uniquely in this way. Set $\psi(e)=0$ for $e\notin K$, set
$\psi(e)=1$ if the chosen orientation of $e$ agrees with the fixed orientation
of $F$, and set $\psi(e)=-1$ otherwise.

It remains to compute the coefficient of an admissible
choice of the integers $a_{v,e}$ satisfying \eqref{eq:sum_v} and \eqref{eq:sum_e}. Its contribution to the coefficient of $M$ is given by the multinomial identity 
\[
\prod_{v\in V(F)}\binom{p-1}{(a_{v,e})_{e\ni v}}=\prod_{v\in V(F)}(p-1)!\prod_{e\ni v}(a_{v,e}!)^{-1}\equiv(-1)^{|V(F)|}\prod_{v\in V(F)}\prod_{e\ni v}(a_{v,e}!)^{-1} \pmod{p},
\]
where the last congruence is due to Wilson's theorem. For $e=uv$, if $\psi(e)=0$, then
$(a_{u,e},a_{v,e})=(1,1)$. If
$\psi(e)\neq 0$, then $(a_{u,e},a_{v,e})\in\{(2,0),(0,2)\}$. Therefore, if $K=K_\psi$, the
contribution of the corresponding Eulerian orientation of $K$ is given by
\[
(-1)^{|V(F)|}2^{-|E(K)|}.
\]
Summing over all admissible choices of $\psi$ is the same as summing over all
Eulerian orientations of all Eulerian subgraphs $K\subseteq F$. Thus the
coefficient of $M$ in \eqref{eq:polynomial-for-M} is
\[
(-1)^{|V(F)|}
\sum_{\substack{K\subseteq F\\ K\text{ Eulerian}}}
EO(K)\,2^{-|E(K)|}
\equiv 
(-1)^{|V(F)|}\, \Psi_p(F) \pmod{p}.
\]

We conclude that $G$ is Hartke
$\mathbb Z_p$-magic whenever
$\Psi_p(F)\not\equiv 0\pmod p$, as claimed.
\end{proof}

\begin{proposition}\label{prop:Kp-count}
Let $p \geq 3$ be prime. Then,
\[
\Psi_p(K_p) \equiv 2^{-\binom{p}{2}} \cdot 2^{p-1} \not\equiv 0 \pmod{p},
\]
where $\Psi_p$ is as defined in Lemma \ref{lem:factor-eulerian-coeff} and $K_p$ is the complete graph on $p$ vertices.
\end{proposition}
\begin{proof}
Identify the vertex set $V(K_p)$ with $\mathbb{Z}_p$. Let
\[
\overrightarrow E
=
\{(a,b)\in \mathbb{Z}_p^2:a\neq b\}
\]
be the directed edge set of the bidirected complete graph. Consider the collection
\[
\mathcal{B}
=
\{A \subseteq \overrightarrow E : d_A^+(a)=d_A^-(a)
 \text{ for all } a \in \mathbb{Z}_p \}.
\]
We first relate $|\mathcal B|$ to $\Psi_p(K_p)$. For each unordered edge
$\{a,b\}$, a subset $A\subseteq \overrightarrow E$ chooses one of the four
possibilities:
\[
\{(a,b)\},\qquad \{(b,a)\},\qquad \{(a,b),(b,a)\}, \qquad  \varnothing.
\]
The first two choices above form an
oriented subgraph of $K_p$. The condition $d_A^+(a)=d_A^-(a)$ for all $a$ ensures that such an orientation
is an Eulerian orientation of an Eulerian subgraph $K\subseteq K_p$. The last two choices remain for each
edge not belonging to $K$.  Hence,
\[
|\mathcal B|
=
\sum_{\substack{K\subseteq K_p\\ K\text{ Eulerian}}}
EO(K)\,2^{\binom p2-|E(K)|},
\]
which implies
\[
2^{-\binom p2}|\mathcal B|
=
\sum_{\substack{K\subseteq K_p\\ K\text{ Eulerian}}}
EO(K)\,2^{-|E(K)|}
=
\Psi_p(K_p).
\]
We now count $|\mathcal{B}|$. For all $t \in \mathbb{Z}_p$, define the translation action by 
\[
t + A=\{(a+t,b+t):(a,b)\in A\}.
\]
Every orbit of this action has size either $1$ or $p$.
Thus, modulo $p$, $|\mathcal B|$ is congruent to the number of fixed
points of this action. Note that we have $A= t+A$ for every $t\in \mathbb{Z}_p$ if and only if 
\[
A= A_D =\{(a,a+d):a\in\mathbb{Z}_p,\ d\in D\},
\]
for some $D\subseteq \mathbb{Z}_p^*$.
Therefore, 
\[
|\mathcal{B}| \equiv \#\{D:D\subseteq \mathbb{Z}_p^*\} = 2^{p-1}.
\]
We conclude
\[
\Psi_p(K_p) \equiv 2^{-\binom p2} \; 2^{p-1} \not\equiv 0 \pmod{p},
\]
as claimed.
\end{proof}

\begin{theorem} \label{thm:Kp-factors} Let $p\geq 5$ be prime.  If a graph $G$ contains a spanning subgraph whose connected
components are copies of $K_p$, then $G$ is Hartke
$\mathbb{Z}_p$-magic.   
\end{theorem}
\begin{proof}
Observe that if $G$ contains a spanning subgraph $F$ such that $F$ is a disjoint union of $s$ copies of $K_p$, then $\Psi_p(F) = (\Psi_p(K_p))^s \not\equiv 0$ by Proposition \ref{prop:Kp-count}; the claim then follows from Lemma \ref{lem:factor-eulerian-coeff}.
\end{proof}

The result of Theorem \ref{thm:Kp-factors}, together with Proposition \ref{prop:induction}, provides the following additional evidence for Conjecture \ref{Conj1}.

\begin{theorem}\label{thm:major-conjecture-support}
Let $p \geq 5$ be prime. For every integer $n \geq p$, there exists a connected simple Hartke $\mathbb{Z}_p$-magic graph of order $n$. 
\end{theorem}
\begin{proof}
Fix a prime $p \geq 5$. By Theorem \ref{thm:Kp-factors}, the complete graph $K_p$ is Hartke $\mathbb{Z}_p$-magic. Now, we may attach a new vertex to any $p-1$ vertices of $K_p$ to obtain another Hartke $\mathbb{Z}_p$-magic graph, by Proposition \ref{prop:induction}. Continuing this process yields the result.
\end{proof}

\begin{remark} \label{rmk:counting-eulerian}
We have computed $\Psi_p$ for the complete graph $K_p$, but it is plausible that Lemma \ref{lem:factor-eulerian-coeff} is applicable to other $(p-1)$-regular graphs. Although we do not pursue this idea here, one may consider the following formulation. Let $F$ be a graph and let $p\geq 3$ be prime.
For a Laurent polynomial $P \in \mathbb{Z}_p[z_v,z_v^{-1}:v\in V(F)]$, let $\operatorname{CT}(P)$ denote its constant
term. Then,
\[
\Psi_p(F)
=
\operatorname{CT}
\prod_{uv\in E(F)}
\left(
1+\frac{1}{2}\frac{z_v}{z_u}
+\frac{1}{2}\frac{z_u}{z_v}
\right).
\]
Indeed, for each edge $uv$, the three terms correspond to not choosing the edge, choosing the orientation $u\to v$, and choosing the orientation $v\to u$. The exponent of $z_v$ records $d^-(v)-d^+(v)$. Hence, taking the constant term imposes $d^-(v)=d^+(v)$ for all $v \in V(F)$. Equivalently, the chosen orientation is an Eulerian orientation of an Eulerian
subgraph $K\subseteq F$. Each chosen oriented edge contributes a
factor of $\frac{1}{2}$, so an Eulerian orientation of $K$ contributes
$2^{-|E(K)|}$. Summing over all such orientations gives the identity.
\end{remark}

As noted in Remark \ref{rmk:hartke-cycle-not-applicable}, Hartke $\mathbb{Z}_p$-magic graphs must satisfy the edge-count bound in Proposition \ref{Prop_1}, which does not allow us to analyze cycle graphs. One may wonder why it is worthwhile to examine cycles, since they are easily seen to be $\mathbb{Z}_p$-magic by assigning a nonzero constant to all edges. However, if one can exhibit the edge-stability property as in statement (2) of Theorem \ref{thm:hartke-special}, results for cycles may extend to all Hamiltonian graphs. We pursue this direction next.

\subsection{Alon $\mathbb{Z}_p$-magic graphs} \label{sec:alon-magic}
Let the edges of a simple graph $G$ with $m$ edges be identified with distinct variables $x_1, x_2, \dots, x_m$. Define the polynomial $g_G$ in $\mathbb{Z}_p[x_1, \dots, x_m]$ as
\begin{align} \label{Aloneq}
g_G(x)=g_G(x_{1},\dots,x_{m})=\prod_{v\in V(G)\setminus\{v^{*}\}}\left[1-\left(\sum_{e\ni v^{*}}x_{e}-\sum_{e\ni v}x_{e}\right)^{p-1}\right],
\end{align}
where addition and multiplication are taken modulo $p$. We call the polynomial $g_G$ the \textit{Alon polynomial}. The vertex $v^*$ is an arbitrarily chosen vertex in $V(G)$. Observe that a graph $G$ is $\mathbb{Z}_p$-magic if and only if there exists $s\in (\mathbb{Z}_p^*)^m$ such that $g_G(s) \neq 0$. While the Hartke polynomial $f_{G,t}$ compares the induced sum at each vertex to a chosen magic value $t$, the Alon polynomial compares the induced sum at a chosen vertex with the induced sums at all of the remaining vertices.

\begin{definition}
Let $p \geq 3$ be prime. A graph $G$ is called \textit{Alon} $\mathbb{Z}_p$-magic \underline{if} the Combinatorial Nullstellensatz can be applied to the Alon polynomial $g_G$ of $G$ to prove that $G$ is $\mathbb{Z}_p$-magic. Equivalently, $g_{G}$ has a degree $(p-1)(|V(G)|-1)$ monomial $M$ with nonzero coefficient whose exponent in each variable is at most $p-2$. In this case, such a monomial $M$ is called an
\textit{Alon term}. 
\end{definition}

\begin{proposition}\label{Prop2}
Let $p \geq 3$ be prime. If $G$ is an Alon $\mathbb{Z}_p$-magic graph, then $|E(G)| \geq \frac{p-1}{p-2} (|V(G)|-1)$. 
\end{proposition}
\begin{proof}
As in Proposition \ref{Prop_1}, this result is a consequence of the degree restriction imposed by the Combinatorial Nullstellensatz. A possible Alon term can have exponent at most $p-2$ in each variable $x_1,\ldots,x_m$. As the degree of an Alon term must be equal to the degree of the Alon polynomial, it follows that we must have $(p-2)|E(G)|\geq (p-1)(|V(G)|-1)$. Therefore, $|E(G)| \geq \frac{p-1}{p-2} (|V(G)|-1)$, as claimed.
\end{proof}

Note that if $p> n$, then we have $n\geq\frac{p-1}{p-2}(n-1)$. One might hope that $C_n$ is  Alon $\mathbb{Z}_p$-magic for $p>n$. However, Theorems \ref{thm:alon-Zp} and \ref{thm:alon-Zp-not} prove that $C_n$ is Alon $\mathbb{Z}_p$-magic for $p>n$ if and only if $n$ is odd. These findings have broader significance, as we will discuss.

\begin{definition}
Let $f\in \mathbb{F}[x_1,\ldots,x_m]$ and $\alpha = (\alpha_1, \alpha_2, \dots, \alpha_m) \in \mathbb{N}^m$. The \textit{monomial corresponding to} $\alpha$ is defined to be $x^\alpha \coloneq x_1^{\alpha_1} \cdots x_m^{\alpha_m}$, and the \textit{degree} of $x^\alpha$ is denoted $|\alpha| = \alpha_1+\cdots+\alpha_m$. Furthermore, the coefficient of a monomial $M$ in $f$ is denoted by $[M]f$.
\end{definition}

\begin{lemma}
\label{lem:path-coeff}
Consider the polynomial
\[
P(z_1,\dots,z_n):=\prod_{j=1}^{n-1}(z_j-z_{j+1})^{\alpha_j} \in \mathbb{Z}_p[z_1,\ldots,z_n],
\]
where $p$ is a prime such that $p>n$, and $\alpha_1, \ldots, \alpha_{n-1}$ are nonnegative integers satisfying \linebreak $\alpha_1+\cdots+\alpha_{n-1}=(n-1)(p-1)$.  If $[z_1^{n-1}z_2^{p-2}\cdots z_n^{p-2}]P \not\equiv 0 \pmod p,
$
then $\alpha_1=\cdots=\alpha_{n-1}=p-1$.
Moreover, in this case, $[z_1^{n-1}z_2^{p-2}\cdots z_n^{p-2}]P \equiv 1 \pmod p$.
\end{lemma}

\begin{proof}
Using the binomial identity, we expand each factor
\[
(z_j-z_{j+1})^{\alpha_j}
=
\sum_{t_j=0}^{\alpha_j}
\binom{\alpha_j}{t_j}(-1)^{t_j}
z_j^{\,\alpha_j-t_j}z_{j+1}^{\,t_j}.
\]
To obtain the monomial $z_1^{\,n-1}z_2^{\,p-2}\cdots z_n^{\,p-2},
$
the exponents of $z_1,\ldots,z_{n-1}$ must satisfy
\[
\alpha_1-t_1 = n-1 \quad \text{and} \quad
t_{j-1}+(\alpha_j-t_j) = p-2 \text{ for all }2\le j\le n-1.
\]
Therefore, we have
\[
0\le t_j<p
\quad\text{and}\quad
0\le \alpha_j-t_j<p \text{ for all } j \in [n-1].
\]
If $\alpha_j\ge p$, then $\alpha_j=p+r$ for some $0\le r\le p-2$, and $t_j>r$. Then,
\[
\binom{\alpha_j}{t_j}
=
\binom{p+r}{t_j}
\equiv
\binom{r}{t_j}
\equiv 0 \pmod p,
\]
by Lucas' theorem. Therefore, if $[z_1^{n-1}z_2^{p-2}\cdots z_n^{p-2}]P \not\equiv 0 \pmod p
$, then we must have $\alpha_j\le p-1$ for all
$j$. Since $\alpha_1+\cdots+\alpha_{n-1}=(n-1)(p-1)$,
we must have $\alpha_1=\cdots=\alpha_{n-1}=p-1,$ as claimed.

It remains to compute the coefficient in this case. Again, we expand each factor
\[
(z_j-z_{j+1})^{p-1}
=
\sum_{t_j=0}^{p-1}\binom{p-1}{t_j}(-1)^{t_j}
z_j^{\,p-1-t_j}z_{j+1}^{\,t_j} \equiv \sum_{t_j=0}^{p-1}
z_j^{\,p-1-t_j}z_{j+1}^{\,t_j} \pmod{p}.
\] 
To obtain the monomial $z_1^{\,n-1}z_2^{\,p-2}\cdots z_n^{\,p-2}$, the exponents must satisfy
\[
p-1-t_1 = n-1 \quad  \text{and } \quad
t_{j-1}+(p-1-t_j) = p-2 \text{ for all }2\le j\le n-1.
\]
The second equation gives $t_j=t_{j-1}+1$, so we obtain a unique solution
\[
t_j=p-n+j-1 \text{ for all }1\le j\le n-1.
\]
Since $p>n$, each $t_j$ lies in $\{0,1,\dots,p-1\}$. Hence, the stated monomial appears with coefficient $1$, as claimed.
\end{proof}

\begin{theorem}\label{thm:alon-Zp}
Let $n\geq 3$ be odd, and let $p>n$ be prime. Then, $C_n$ is Alon
$\mathbb Z_p$-magic.
\end{theorem}

\begin{proof}
Denote the cycle $C_n$ by $v_1v_2\cdots v_nv_1$. Let $x_1$ correspond to the edge $v_nv_1$, and let $x_i$ correspond to the
edge $v_{i-1}v_i$ for $2\leq i\leq n$. We use $v_n$ as the distinguished
vertex. The highest-degree monomials of the corresponding Alon polynomial must come from the polynomial
\begin{equation}\label{eq:cycle-alon-poly-odd}
    f=
\prod_{i=1}^{n-1}
\left((x_1+x_n)-(x_i+x_{i+1})\right)^{p-1}.
\end{equation}
Set $x_0:=x_n$ and define
\[
y_i:=x_{i-1}-x_{i+1} \text{ for all }1\leq i\leq n-1.
\]
Observe that for each $1 \leq i \leq n-1$, we have
\[
(x_1+x_n)-(x_i+x_{i+1})
=
\sum_{t=1}^{i}(x_{t-1}-x_{t+1})
=
y_1+\cdots+y_i.
\]
Hence, we can rewrite \eqref{eq:cycle-alon-poly-odd} as
\begin{equation}\label{eq:cyle-proof-1}
    f=\prod_{i=1}^{n-1}(y_1+\cdots+y_i)^{p-1}.
\end{equation}
We first note that
\[
[y_1^{p-1}\cdots y_{n-1}^{p-1}]\,f=1.
\]
This follows by choosing $y_i^{p-1}$ from the $i$-th factor $(y_1+\cdots+y_i)^{p-1}$ for all $i\in [n-1]$.
 Since $n$ is odd, write $n=2m-1$. Define a bijection
$\sigma:[n]\to [n]$ by
\[
\sigma(j)=
\begin{cases}
\dfrac{j+1}{2}, & \text{if } j \text{ is odd},\\[6pt]
m+\dfrac{j}{2}, & \text{if } j \text{ is even}.
\end{cases}
\]
Now relabel the edge variables by setting
\[
z_{\sigma(j)}:=x_j \text{ for all } j\in [n].
\]
Then, we have
\[
\{y_1,y_2,\ldots,y_{n-1}\}=\{z_1-z_2, \ z_2-z_3,\ldots, z_{n-1}-z_n \},
\]
which gives
\begin{equation}\label{eq:cyle-proof-2}
    y_1^{p-1}\cdots y_{n-1}^{p-1}
=
\prod_{j=1}^{n-1}(z_j-z_{j+1})^{p-1}.
\end{equation}
By Lemma \ref{lem:path-coeff}, the monomial
\[
M := z_1^{n-1}z_2^{p-2}\cdots z_n^{p-2} = x_1^{n-1}x_2^{p-2}\cdots x_n^{p-2}
\]
appears in \eqref{eq:cyle-proof-2} with coefficient $1$ modulo $p$. In fact, by Lemma \ref{lem:path-coeff}, no other term in the expanded form of \eqref{eq:cyle-proof-1} can contribute nontrivially to the monomial $M$, meaning the coefficient of $M$ in $f$ is 1 modulo $p$. Since $n-1\le p-2$ and all other exponents are equal to $p-2$, this monomial $M$ is a valid Alon term, which establishes the result. 
\end{proof}

It is natural to ask whether the special properties of Hartke $\mathbb{Z}_p$-magic graphs stated in Theorem \ref{thm:hartke-special} also hold for Alon $\mathbb{Z}_p$-magic graphs. In this setting, the results are mixed.

\begin{proposition} \label{thm:alon-stable}
Let $p \geq 3$ be prime, and let $G$ be an Alon $\mathbb{Z}_p$-magic graph. Then, a graph $G'$ formed by adding simple edges to $G$ is also Alon $\mathbb{Z}_p$-magic.
\end{proposition}
\begin{proof}
Let $G$ be a graph, and fix a vertex $v^*\in V(G)$. Consider the Alon polynomial
\[
g_G=\prod_{v\in V(G)\setminus\{v^{*}\}}\left[1-\left(\sum_{e\ni v^{*}}x_{e}-\sum_{e\ni v}x_{e}\right)^{p-1}\right]
\]
corresponding to $G$. Suppose $G$ is Alon $\mathbb{Z}_p$-magic. Then, there exists an Alon term
\[
M=c\prod_{i=1}^m x_i^{t_i}
\]
in the expansion of $g_G$ such that $\deg(M)=(p-1)(n-1),
 c\not\equiv 0 \pmod{p},$ and $t_i\le p-2$ for all $i$. Now, let $G'$ be obtained from $G$ by adding simple edges. Since $V(G')=V(G)$, we have $\deg(g_{G'})=(p-1)(n-1)=\deg(g_G).$ For each vertex $v\neq v^*$, denote by $L_v(x_1,\dots,x_m)$ the linear form appearing in the corresponding canonical factor of $g_G$. In $g_{G'}$, the analogous linear form has the form 
\[
L_v(x_1,\dots,x_m)+N_v(x_{m+1},\dots),
\]
where $N_v$ is a linear form involving only the new variables.
Hence, the $v$-th canonical factor of $g_{G'}$ is
\[
1-\bigl(L_v+N_v\bigr)^{p-1}.
\]
The term $-L_v^{\,p-1}$ in this factor appears
with the same coefficient as in the $v$-th canonical factor of $g_{G}$, i.e., $1-L_v^{\,p-1}$.
Thus, from each canonical factor of $g_{G'}$, we can choose the same summand as was chosen from the corresponding canonical factor of $g_G$ in
the construction of the monomial $M$. Therefore, $M$ is an Alon term for $G'$, from which the result follows.
\end{proof}

\begin{remark}
We note that the claim of Proposition \ref{thm:alon-stable} remains true when adding multiedges or self-loops; the same holds for statement (2) of Theorem \ref{thm:hartke-special}. The remaining special properties of Hartke $\mathbb{Z}_p$-magic graphs as stated in Theorem \ref{thm:hartke-special} are not necessarily true for Alon $\mathbb{Z}_p$-magic graphs. The key ingredient used in \cite{Low_Roberts_Constructing} to construct a larger class of $\mathbb{Z}_p$-magic graphs via graph products was to show that a disjoint union of  Hartke $\mathbb{Z}_p$-magic graphs is again Hartke $\mathbb{Z}_p$-magic. However, a similar claim does not hold for Alon $\mathbb{Z}_p$-magic graphs. For instance, $C_3$ is Alon $\mathbb{Z}_5$-magic, but $C_3 \sqcup C_3$ violates the edge-count condition in Proposition \ref{Prop2} for $p=5$.
\end{remark}

Combining Proposition \ref{thm:alon-stable} with Theorem \ref{thm:alon-Zp} gives the following important result.

\begin{corollary}\label{cor:odd-hamiltonian}
    Let $G$ be a Hamiltonian graph, where $|V(G)| = n$ is odd. Then, $G$ is $\mathbb{Z}_p$-magic, for all primes $p > n$.
\end{corollary}

Small examples show that Hamiltonian graphs on an even number of vertices are not necessarily $\mathbb{Z}_p$-magic. For example, as pointed out in Remark \ref{rmk:edge-stability-c4+e}, $C_4 + \{e\}$ is not $\mathbb{Z}_3$-magic. It turns out that even cycles are not Alon $\mathbb Z_p$-magic, which we now show.

\begin{lemma}\label{lem:Lagrange}
Let $p$ be prime, and fix
$\alpha=(\alpha_1,\dots,\alpha_n)$ with $0\leq \alpha_i\leq p-1$ for all
$i\in[n]$. Let $f\in \mathbb Z_p[x_1,\dots,x_n]$ be a homogeneous polynomial of degree
$|\alpha|=\alpha_1+\cdots+\alpha_n$. Then,
\[
[x^\alpha]f
=
(-1)^n
\sum_{z\in\mathbb Z_p^n}
f(z)\prod_{i=1}^n z_i^{p-1-\alpha_i}.
\]
\end{lemma}

\begin{proof}
As $f$ is homogeneous of degree $|\alpha|$, we may write
\[
f(x)=\sum_{|\beta|=|\alpha|} c_\beta \,x^\beta .
\]
Then,
\begin{equation} \label{eq:lem-coeff-homog}
    \sum_{z\in\mathbb Z_p^n}
f(z)\prod_{i=1}^n z_i^{p-1-\alpha_i}
=
\sum_{|\beta|=|\alpha|} c_\beta
\prod_{i=1}^n
\left(
\sum_{z_i\in\mathbb Z_p}
z_i^{\beta_i+p-1-\alpha_i}
\right).
\end{equation}
We use the standard identity
\[
\sum_{z\in\mathbb Z_p} z^r
=
\begin{cases}
-1, & r>0 \text{ and } (p-1)\mid r,\\
0, & \text{otherwise}.
\end{cases}
\]
The summand on the right-hand side of \eqref{eq:lem-coeff-homog} indexed by $\beta$ can be nonzero only if $\beta_i+p-1-\alpha_i
$ is a positive multiple of $p-1$ for all  $i \in [n]$. Equivalently, $\beta_i=\alpha_i+q_i(p-1)
$ for some integer $q_i\geq 0$. Since $|\beta|=|\alpha|$, we get
\[
0=|\beta|-|\alpha|
=(p-1)\sum_{i=1}^n q_i.
\]
This implies we must have $q_i=0$ for all $i\in[n]$, meaning $\beta=\alpha$. Therefore, we obtain
\[
\sum_{z\in\mathbb Z_p^n}
f(z)\prod_{i=1}^n z_i^{p-1-\alpha_i}
=
c_\alpha(-1)^n.
\]
Since $c_\alpha=[x^\alpha]f$, multiplying by $(-1)^n$ gives the result.
\end{proof}

\begin{theorem} \label{thm:alon-Zp-not}
Let $n\ge 4$ be even, and let $p>n$ be prime. Then, $C_n$ is not Alon $\mathbb Z_p$-magic.
\end{theorem}
\begin{proof}
Denote the cycle $C_n$ by $v_1v_2\cdots v_nv_1$. Let $x_1$ correspond to the edge $v_nv_1$, and let $x_i$ correspond to the
edge $v_{i-1}v_i$ for $2\leq i\leq n$. We use $v_n$ as the distinguished
vertex. The highest-degree monomials of the corresponding Alon polynomial must come from the polynomial
\begin{equation}\label{eq:cycle-alon-poly-even}
    f=
\prod_{i=1}^{n-1}
\left((x_1+x_n)-(x_i+x_{i+1})\right)^{p-1}.
\end{equation}
Observe that $f$ is a homogeneous polynomial of degree $(n-1)(p-1)$.

Suppose, toward a contradiction, that $C_n$ is Alon $\mathbb Z_p$-magic,  where $n$ is even and $p > n$ is prime. Then there exists a monomial $x^\alpha=x_1^{\alpha_1}\cdots x_n^{\alpha_n}$
appearing in $f$ with nonzero coefficient and satisfying
$0\leq \alpha_i\leq p-2$, for all $i \in [n]$. Further, we have 
\[
\sum_{i=1}^n \alpha_i=(n-1)(p-1).
\]
Define
\[
\gamma_i=p-1-\alpha_i.
\]
Then, $\gamma_i\geq 1$ for all $i$, and $\sum_{i=1}^n \gamma_i=p-1$.
By Lemma \ref{lem:Lagrange}, we have
\begin{equation} \label{eq:Lagrange}
    [x^\alpha]f
=
(-1)^n
\sum_{z\in\mathbb{Z}_p^n}
f(z)\prod_{i=1}^n z_i^{\gamma_i}.
\end{equation}
For $a,b\in\mathbb{Z}_p$, consider the translation
\[
z_i\mapsto z_i+a \quad \text{for }i\text{ odd},\qquad
z_i\mapsto z_i+b \quad \text{for }i\text{ even}.
\]
Since $n$ is even, every vertex sum is translated by $a+b$. Hence,
each difference of vertex sums in \eqref{eq:cycle-alon-poly-even} is unchanged, meaning $f(z)$ is invariant under this
translation. Let $\mathcal R$ be a set containing exactly one representative from each
orbit of this action. We may partition the sum in \eqref{eq:Lagrange} as
\[
[x^\alpha]f
=
(-1)^n
\sum_{z\in\mathcal R}
f(z)
\left(\sum_{a\in\mathbb{Z}_p}\prod_{i\text{ odd}}(z_i+a)^{\gamma_i}\right)
\left(\sum_{b\in\mathbb{Z}_p}\prod_{i\text{ even}}(z_i+b)^{\gamma_i}\right).
\]
Observe that
\[
\prod_{i\text{ odd}}(z_i+a)^{\gamma_i}
\]
is a polynomial in $a$ of degree 
\[
0<\sum_{i\text{ odd}}\gamma_i<p-1,
\]
since every $\gamma_i\geq 1$ and $\sum_i\gamma_i=p-1$.
Therefore, we must have
\[
\sum_{a\in\mathbb{Z}_p}\prod_{i\text{ odd}}(z_i+a)^{\gamma_i} \equiv 0 \pmod{p}.
\]
A similar argument for $\displaystyle \prod_{i\text{ even}}(z_i+b)^{\gamma_i}
$ establishes that every orbit contributes zero, and hence
\[
[x^\alpha]f \equiv 0 \pmod{p},
\]
contradicting the choice of $x^\alpha$. Therefore, $C_{n}$ is not Alon
$\mathbb Z_p$-magic when $n$ is even.
\end{proof}

\subsection{General $\mathbb{Z}_k$-magic graphs}

The Combinatorial Nullstellensatz is formulated over a field; hence, working over the finite field $\mathbb{Z}_p$ is convenient for $\mathbb{Z}_p$-magic graphs, where $p$ is prime. For general $\mathbb{Z}_k$-magic graphs, we may use roots of unity to encode modular sums and obtain a polynomial certificate.

\begin{proposition}\label{prop:general-Zn}
Let $G$ be a graph. Suppose that $k\ge 2$, $t\in \mathbb{Z}_k$, and $\omega$ is a primitive $k$-th root of unity. Consider the polynomial
\[
P_{G,t}(x)
=
\prod_{v\in V(G)}
\frac{1}{k}
\sum_{j=0}^{k-1}
\omega^{-tj}
\prod_{e\ni v}x_e^{j} \; \in \mathbb{C}[x],
\]
where ${x}=(x_e)_{e\in E}$.
For an edge labeling $h:E(G)\to \mathbb{Z}_k^*,
$
write $P_{G,t}(h):=P_{G,t}\bigl((\omega^{h(e)})_{e\in E(G)}\bigr)$.
Then,
\[
P_{G,t}(h)=
\begin{cases}
1, & \text{if } h \text{ is a } t\text{-sum } \mathbb{Z}_k\text{-magic labeling},\\
0, & \text{otherwise.}
\end{cases}
\]
\end{proposition}

\begin{proof}
For each vertex $v\in V(G)$, set
\[
S_v=\sum_{e\ni v}h(e)\in \mathbb{Z}_k.
\]
Then, the factor of $P_{G,t}(h)$ corresponding to $v$ is
\[
\frac{1}{k}
\sum_{j=0}^{k-1}
\omega^{-tj}
\prod_{e\ni v}\omega^{j h(e)}
=
\frac{1}{k}
\sum_{j=0}^{k-1}
\omega^{j(S_v-t)}.
\]
If $h$ is a $t$-sum $\mathbb{Z}_k$-magic labeling, then $S_v\equiv t \pmod k
$ for every vertex $v$. Hence, $\omega^{j(S_v-t)}=1
$
for every $j$, and so each vertex factor is
\[
\frac{1}{k}\sum_{j=0}^{k-1}1=1,
\]
which gives $P_{G,t}(h)=\prod_{v\in V(G)}1=1$.

Conversely, suppose that $h$ is not a $t$-sum $\mathbb{Z}_k$-magic labeling. Then, there exists at least one vertex $v$ such that $S_v\not\equiv t \pmod k$. For this vertex, $\omega^{S_v-t}\neq 1$. But since $\omega^{k(S_v-t)}=1$, we have
\[
\sum_{j=0}^{k-1}\omega^{j(S_v-t)}
=
\frac{1-\omega^{k(S_v-t)}}{1-\omega^{S_v-t}}
=
0.
\]
Hence, the factor corresponding to this vertex is $0$, and $P_{G,t}(h)=0$.
\end{proof}

The degree of the polynomial $P_{G,t}$  in Proposition \ref{prop:general-Zn} is  $2|E(G)|(k-1)$. Hence, the methods developed for $\mathbb{Z}_p$-magic graphs in earlier sections of this paper do not transfer directly. In this case, if the reduction of $P_{G,t}$ modulo the ideal generated by $\{1+x_e+\cdots +x_e^{k-1}:e\in E(G)\}$ is not identically zero, then $G$ is $t$-sum $\mathbb{Z}_k$-magic. We leave this topic for future work beyond explicit computations for small graphs. 

\section{$\mathbb{Z}_{k}$-antimagic labelings} \label{S:antimagic}

Recall that a graph $G$ is called {$\mathbb{Z}_{k}$-antimagic} if there exists an edge labeling $h: E(G) \to \mathbb{Z}_{k} \backslash \{0\}$ such that the induced vertex labeling $h^+(v)=\sum_{e\ni v} h(e)$ (mod $k$) is injective. 

\begin{lemma}\label{power2}
Let $G$ be a connected simple graph with $m\geq 2$ edges. Then, $G$ is $\mathbb{Z}_{k}$-antimagic for all $k \geq 2^m - 1$.
\end{lemma}
\begin{proof}
Label the edges with the distinct values $2^j$ for $j = 0, 1, \dots, m-1$. Then, the induced label on any vertex is a sum of a nonempty subset of these powers of 2 and therefore has a value between $1$ and $2^m-1$. All such vertex labels are distinct because no two such sums can be equal unless they include exactly the same summands. Thus, the vertex labels remain distinct when reduced modulo $k$, for any $k \geq 2^m-1$.
\end{proof}

When $k=p$ is prime, we can substantially improve the bound in Lemma \ref{power2}.

\begin{theorem}\label{thm:weak-bound-connected-graph}
Let $n \geq3$. Every connected graph on $n$ vertices is $\mathbb{Z}_p$-antimagic whenever $p$ is prime and $p>2n-3$.
\end{theorem}

\begin{proof}
Let $G = (V,E)$ be a connected graph on $n$ vertices. For each edge
$e\in E(G)$, let $x_e$ be the corresponding variable. For each vertex
$v\in V(G)$, define
\[
s_v=\sum_{e\ni v}x_e.
\]
Consider the polynomial
\begin{equation}\label{eq:alon-antimagic}
    p_G=\prod_{\{u,v\}\subseteq V(G)}(s_u-s_v)
\in \mathbb{Z}_p[x_1,\ldots,x_m].
\end{equation}
For any labeling $h:E(G)\to \mathbb{Z}_{p} \backslash \{0\}$, write $p_G(h) := p_G\left((h(e))_{e\in E(G)}\right)$. Then $p_G(h) \neq 0$ if and
only if the induced vertex sums
\[
h^+(v)=\sum_{e\ni v}h(e)
\]
are pairwise distinct, i.e., $h$ is a $\mathbb{Z}_p$-antimagic labeling. The result follows if the polynomial $p_G$ in \eqref{eq:alon-antimagic} has a monomial whose coefficient is nonzero and whose exponent in each variable is at most $2n-4$. We proceed to establish this.

First, note that each factor $(s_u-s_v)$ is nonzero. Indeed, if $u$ and $v$ are nonadjacent, then, since $G$ is connected, there is a shortest $u$-$v$ path. The first edge on this path is incident to $u$ but not to $v$, so its variable appears in $(s_u-s_v)$ with nonzero coefficient. If $u$ and $v$ are adjacent, then the variable corresponding to $uv$ cancels, but since $G$ is connected and $n\geq 3$, at least one of $u$ and $v$ is incident to another edge, whose variable appears in exactly one of $s_u$ and $s_v$. Hence, $s_u-s_v\neq 0$. It follows that the polynomial $p_G$ is not identically zero. Moreover, $p_G$ is homogeneous of degree $\binom n2$. Hence, there
exists a monomial
\[
x^\alpha=\prod_{e\in E(G)}x_e^{\alpha_e}
\]
of total degree $\binom n2$ whose coefficient in $p_G$ is nonzero. In fact, choosing the leading monomial of $p_G$ with respect to any monomial order gives a monomial whose coefficient is $\pm 1$. Indeed, fix any monomial order. Since each factor $s_u-s_v$ is a nonzero
linear form with coefficients in $\{0,1,-1\}$, its leading term has coefficient
$\pm1$. By multiplicativity of leading terms, the leading term of $p_G$ is the
product of the leading terms of these factors. We now bound the exponents in such a monomial.

Fix an edge $e=ab$. The
variable $x_e$ appears in the factor $s_u-s_v$ only when exactly one of
$u,v$ lies in $\{a,b\}$. There are at most $2(n-2)
$ such pairs. Therefore, no monomial of $p_G$ can contain $x_e$ with exponent
larger than $2(n-2)$. Thus, we have
\[
\alpha_e\leq 2n-4
\text{ for all }e\in E(G).
\]
Hence, by the Combinatorial
Nullstellensatz, there exists a vector $\textbf{x} \in (\mathbb{Z}_p^*)^m$ such that $p_G(\textbf{x})\neq 0$, provided that $p > 2n-3$. This implies that $G$ is $\mathbb{Z}_p$-antimagic.
\end{proof}

\begin{remark} \label{rmk:antimagic-stable}
In \cite{Odab_Roberts_Low}, it is shown that if a connected graph $G$ is $\mathbb{Z}_k$-antimagic, then a graph $G'$ formed by adding simple edges to $G$ is also $\mathbb{Z}_k$-antimagic. The authors of \cite{Odab_Roberts_Low} did not use polynomial methods. Applying the same reasoning used in the proof of Proposition \ref{thm:alon-stable} to the polynomial construction in \eqref{eq:alon-antimagic}, we can recover their result for the class of graphs for which the Combinatorial Nullstellensatz certifies the $\mathbb{Z}_p$-antimagic property.
\end{remark}

\begin{remark}
We may drop the condition ``modulo $k$'' from the definition of a $\mathbb{Z}_k$-antimagic labeling and call the resulting labeling a \textit{$k$-antimagic labeling}. Using the same argument as in the proof of Theorem \ref{thm:weak-bound-connected-graph}, but over the polynomial ring $\mathbb{R}[x_1,\ldots,x_m]$, we obtain that every connected graph on $n \geq 3$ vertices is $k$-antimagic whenever $k > 2n-3$.
\end{remark}

\section{Discussion and future directions}
In this paper, we used the Combinatorial Nullstellensatz to study
$\mathbb{Z}_p$-magic and $\mathbb{Z}_p$-antimagic  graph labeling problems. We showed that, in several cases, the coefficients of monomials admit
combinatorial interpretations in terms of structural properties of the graph. These interpretations appear to be of
independent interest and suggest several directions for further work.

As pointed out in Remark \ref{rmk:counting-eulerian}, computing $\Psi_p$ for broader classes of $(p-1)$-regular graphs, as defined in Lemma \ref{lem:factor-eulerian-coeff}, remains open. In this paper, we evaluated $\Psi_p(K_p)$ and used this computation to obtain Hartke $\mathbb Z_p$-magic graphs. A particularly interesting next case is the complete bipartite graph $K_{p-1, p-1}$. The computation of $\Psi_p(K_{p-1,p-1})$ has a concrete matrix interpretation. By associating the edge corresponding to row $i$ and column $j$ with the $(i,j)$-entry of a $(p-1) \times (p-1)$ matrix over $\mathbb{Z}_p$, an admissible function
$\psi:E(K_{p-1,p-1})\to\{-1,0,1\}$ in the proof of
Lemma~\ref{lem:factor-eulerian-coeff} may be viewed as a weighted enumeration of balanced matrices with entries in $\{-1,0,1\}$ such that all row and column sums are zero.

Another direction is to further clarify the relationship between ordinary
$\mathbb Z_p$-magic labelings and Hartke and Alon $\mathbb Z_p$-magic certificates. For example, if $G$ is $5$-regular, then the existence of an ordinary $0$-sum
$\mathbb Z_3$-magic labeling is equivalent to the existence of a spanning subgraph $H\subseteq G$ satisfying $d_H(v)\in\{1,4\}$ for every vertex $v$.
Thus, the labeling problem can be reformulated as a
degree-constrained graph factorization problem. It would be interesting to
understand whether such graph factorization problems can be studied via 
Hartke-Alon-type certificates.

Finally, the proof of
Theorem~\ref{thm:weak-bound-connected-graph} for $\mathbb{Z}_p$-antimagic labelings gives a uniform bound for all connected graphs. However, because of the edge-stability property, one may focus on the underlying spanning trees. Refined results may be obtainable if additional parameters, such as the maximum degree, are taken into account. In particular, one may ask for conditions under which the Combinatorial Nullstellensatz recovers bounds closer to those expected from
explicit $\mathbb{Z}_p$-antimagic constructions. 

\section{Tool and computational resource disclosure}\label{Disclosure}
ChatGPT was used to help draw figures and proofread the manuscript for grammar, spelling, and punctuation.

\section{Acknowledgments}
The authors gratefully acknowledge the welcoming atmosphere of the \textit{57th Southeastern International Conference on Combinatorics, Graph Theory \& Computing} (March 9--13, 2026, Florida Atlantic University, Boca Raton, FL, USA), where this collaboration began.

We also thank Professor Edinah Gnang for his helpful feedback on the manuscript.

\end{document}